\documentclass[10pt]{amsart}
\usepackage{amssymb,MnSymbol,times}
\usepackage{amsthm,amsmath,cite,color}

\title[Reducibility of $n$-ary semigroups]{Reducibility of $n$-ary semigroups: from quasitriviality towards idempotency}

\author{Miguel Couceiro}
\address{Universit\'e de Lorraine, CNRS, Inria Nancy G.E., LORIA, F-54000 Nancy, France}
\email{miguel.couceiro[at]\{inria,loria\}.fr}

\author{Jimmy Devillet}
\address{University of Luxembourg, Department of Mathematics, Maison du Nombre, 6, avenue de la Fonte, L-4364 Esch-sur-Alzette, Luxembourg}
\email{jimmy.devillet[at]uni.lu}

\author{Jean-Luc Marichal}
\address{University of Luxembourg, Department of Mathematics, Maison du Nombre, 6, avenue de la Fonte, L-4364 Esch-sur-Alzette, Luxembourg}
\email{jean-luc.marichal[at]uni.lu}

\author{Pierre Mathonet}
\address{University of Li\`ege, Department of Mathematics, All\'ee de la D\'ecouverte, 12 - B37, B-4000 Li\`ege, Belgium}
\email{p.mathonet[at]uliege.be}

\date{June 16, 2020}

\begin{document}

\theoremstyle{plain}
\newtheorem{theorem}{Theorem}[section]
\newtheorem{lemma}[theorem]{Lemma}
\newtheorem{proposition}[theorem]{Proposition}
\newtheorem{corollary}[theorem]{Corollary}
\newtheorem{fact}[theorem]{Fact}
\newtheorem{conjecture}[theorem]{Conjecture}
\newtheorem*{main}{Main Theorem}

\theoremstyle{definition}
\newtheorem{definition}[theorem]{Definition}
\newtheorem{example}[theorem]{Example}
\newtheorem{algorithm}{Algorithm}

\theoremstyle{remark}
\newtheorem{remark}{Remark}
\newtheorem{claim}{Claim}

\newcommand{\N}{\mathbb{N}}
\newcommand{\Z}{\mathbb{Z}}
\newcommand{\R}{\mathbb{R}}
\newcommand{\Cdot}{\boldsymbol{\cdot}}

\begin{abstract}
Let $X$ be a nonempty set. Denote by $\mathcal{F}^n_k$ the class of associative operations $F\colon X^n\to X$ satisfying the condition $F(x_1,\ldots,x_n)\in\{x_1,\ldots,x_n\}$ whenever at least $k$ of the elements $x_1,\ldots,x_n$ are equal to each other. The elements of $\mathcal{F}^n_1$ are said to be quasitrivial and those of $\mathcal{F}^n_n$ are said to be idempotent. We show that $\mathcal{F}^n_1=\cdots =\mathcal{F}^n_{n-2}\subseteq\mathcal{F}^n_{n-1}\subseteq\mathcal{F}^n_n$ and we give conditions on the set $X$ for the last inclusions to be strict. The class $\mathcal{F}^n_1$ was recently characterized by Couceiro and Devillet \cite{CouDev}, who showed that its elements are reducible to binary associative operations. However, some elements of $\mathcal{F}^n_n$ are not reducible. In this paper, we characterize the class $\mathcal{F}^n_{n-1}\setminus\mathcal{F}^n_1$ and show that its elements are reducible. We give a full description of the corresponding reductions and show how each of them is built from a quasitrivial semigroup and an Abelian group whose exponent divides $n-1$.
\end{abstract}



\maketitle

\section{Introduction}

Let $X$ be a nonempty set, let $|X|$ be its cardinality, and let $n\geq 2$ be an integer. An $n$-ary operation $F\colon X^n \to X$ is said to be \emph{associative} if
\begin{multline*}
F(x_1,\ldots,x_{i-1},F(x_i,\ldots,x_{i+n-1}),x_{i+n},\ldots,x_{2n-1})\\
=~ F(x_1,\ldots,x_i,F(x_{i+1},\ldots,x_{i+n}),x_{i+n+1},\ldots,x_{2n-1}),
\end{multline*}
for all $x_1,\ldots,x_{2n-1}\in X$ and all $1\leq i\leq n-1$. The pair $(X,F)$ is then called an \emph{$n$-ary semigroup}. This notion is due to D\"{o}rnte \cite{Dor28} and has led to the concept of $n$-ary group, which was first studied by Post \cite{Pos40}.

In \cite{DudMuk06} the authors investigated associative $n$-ary operations that are determined by binary associative operations. An $n$-ary operation $F\colon X^n\to X$ is said to be \emph{reducible to} an associative binary operation $G\colon X^2\to X$ if there are $G^m\colon X^{m+1}\to X$ ($m=1,\ldots,n-1$) such that $G^{n-1}=F$, $G^1=G$, and
$$
G^m(x_1,\ldots,x_{m+1}) ~=~ G^{m-1}(x_1,\ldots,x_{m-1},G(x_m,x_{m+1})), \qquad m\geq 2.
$$
The pair $(X,F)$ is then said to be the \emph{$n$-ary extension} of $(X,G)$. In that case, we also say that $F$ is the $n$-ary extension of $G$.

Also, an $n$-ary operation $F\colon X^n\to X$ is said to be
\begin{itemize}
\item \emph{idempotent} if $F(x,\ldots,x) =x$ for all $x\in X$,
\item \emph{quasitrivial} \cite{Ack,Lan80} (or \emph{conservative} \cite{PouRosSto96}) if $F(x_1,\ldots,x_n)\in\{x_1,\ldots,x_n\}$ for all $x_1,\ldots,x_n\in X$.
\end{itemize}
Clearly, any quasitrivial $n$-ary operation is idempotent. As we will illustrate below, the converse is not true, even for associative operations. 

The quest for conditions under which an associative $n$-ary operation is reducible to an associative binary operation gained an increasing interest since the pioneering work of Post \cite{Pos40} (see, e.g., \cite{Ack,CouDev,DevKiMar17,DudMuk06,KiSom18,LehSta19}). A necessary and sufficient condition for reducibility was given by Dudek and Mukhin \cite{DudMuk06} using the concept of \emph{neutral element}.
Recall that an element $e\in X$ is said to be \emph{neutral} for $F\colon X^n\to X$ if
\begin{equation}\label{neutralelem}
F((k-1)\Cdot e,x,(n-k)\Cdot e) ~=~ x,\qquad x\in X,~k\in\{1,\ldots,n\}.
\end{equation}
Here and throughout, for any $k\in\{0,\ldots,n\}$ and any $x\in X$, the notation $k\Cdot x$ stands for the $k$-tuple $x,\ldots,x$. For instance, we have
$$
F(3\Cdot x,0\Cdot y,2\Cdot z) ~=~ F(x,x,x,z,z).
$$
Throughout this paper we also denote the set of neutral elements for an operation $F\colon X^n\to X$ by $E_F$. Recall that for any binary operation $G\colon X^2 \to X$ we have $|E_{G}|\leq 1$.

Dudek and Mukhin \cite[Lemma 1]{DudMuk06} proved that if an associative operation $F\colon X^n \to X$ has a neutral element $e$, then it is reducible to the associative operation $G_{e}\colon X^2\to X$ defined by
\begin{equation}\label{eq:dud}
G_{e}(x,y) ~=~ F(x,(n-2)\Cdot e,y),\qquad x,y \in X.
\end{equation}
Furthermore, it was recently observed \cite[Corollary 2.3]{CouDev} that all the quasitrivial associative $n$-ary operations are reducible to associative binary operations. However, there are associative operations that are neither quasitrivial nor reducible to any binary operation; for instance, the associative and idempotent ternary operation $F\colon\mathbb{R}^3 \to \mathbb{R}$ defined by $F(x,y,z) = x-y+z$ (see, e.g., \cite{Sze86} or more recently \cite{MarMat11}).

The observations above show that it is natural to seek conditions under which an idempotent $n$-ary semigroup is reducible to a semigroup. To this extent, we will investigate certain subclasses of idempotent $n$-ary semigroups that contain the quasitrivial ones. In this direction, we will consider classes where the condition
$$
F(x_1,\ldots,x_n)\in \{x_1,\ldots,x_n\}
$$
holds on at least some subsets of $X^n$. More precisely, for a set $S\subseteq \{1,\ldots,n\}$, let
\[
D^n_S ~=~ \{(x_1,\ldots,x_n)\in X^n: ~\forall i,j\in S, x_i=x_j\},
\]
and, for every $k\in\{1,\ldots,n\}$, let
\[
D^n_k ~=~ \bigcup_{\textstyle{S\subseteq\{1,\ldots,n\}\atop |S|\geq k}}D^n_S ~=~ \bigcup_{\textstyle{S\subseteq\{1,\ldots,n\}\atop |S| = k}}D^n_S.
\]
Thus, the set $D^n_k$ consists of those tuples of $X^n$ for which at least $k$ components are equal to each other. In particular, $D_1^n=X^n$ and $D_n^n=\{(x,\ldots,x):x\in X\}$.

For every $k\in\{1,\ldots,n\}$, denote by $\mathcal{F}^n_k$ the class of those associative $n$-ary operations $F\colon X^n\to X$ that satisfy
$$
F(x_1,\ldots,x_n)\in\{x_1,\ldots,x_n\},\quad \text{whenever} \,(x_1,\ldots,x_n)\in D^n_k.
$$
We say that these operations are \emph{quasitrivial on $D^n_k$}.

Thus defined, $\mathcal{F}^n_1$ is exactly the class of quasitrivial associative $n$-ary operations and $\mathcal{F}^n_n$ is exactly the class of idempotent associative $n$-ary operations. It follows directly from the definition of the classes $\mathcal{F}^n_k$ that $\mathcal{F}^n_1 = \mathcal{F}^n_2 = \cdots = \mathcal{F}^n_n$ if $|X| \leq 2$. Therefore, throughout the rest of this paper we assume that $|X| \geq 3$. Since the sets $D^n_k$ are nested in the sense that $D^n_{k+1}\subseteq D_k^n$ for $1\leq k\leq n-1$, the classes $\mathcal{F}^n_k$ clearly form a filtration, that is,
\[
\mathcal{F}^n_1\subseteq\mathcal{F}^n_2\subseteq\cdots\subseteq \mathcal{F}^n_n.
\]

Quite surprisingly, we have the following result, which shows that this filtration actually reduces to three nested classes only.

\begin{proposition}\label{prop:QTDn2}
For every $n\geq 3$, we have $\mathcal{F}^n_1=\mathcal{F}^n_{n-2}$.
\end{proposition}

The proof is deferred until Section \ref{Sect:2}, and so are the proofs of the other results in this introduction.

We observe that the class $\mathcal{F}^n_1=\mathcal{F}^n_2=\cdots=\mathcal{F}^n_{n-2}$ was characterized by Couceiro and Devillet~\cite{CouDev} who showed that all its elements are reducible. More precisely, the following result summarizes \cite[Corollary 3.8]{CouDev} and \cite[Corollary 3.11]{CouDev}.

\begin{proposition}\label{prop:imp}
If $F\colon X^n \to X$ is an associative quasitrivial operation, then $|E_F|\leq 2$ and $F$ has either one or two binary reductions. Furthermore, the binary reductions depend on $E_F$ as follows.
\begin{enumerate}
\item[(a)] If $E_F=\varnothing$, then the operation $G\colon X^2 \to X$ defined for every $x,y\in X$ by
$$
G(x,y)=F((n-1)\Cdot x,y)=F(x,(n-1)\Cdot y)
$$ is the only binary reduction of $F$, and $G$ is quasitrivial.
\item[(b)] If $E_F=\{e\}$, then the operation $G_e\colon X^2 \to X$ defined for every $x,y\in X$ by
$$
G_e(x,y)=F(x,(n-2)\Cdot e,y)=F((n-1)\Cdot x,y)=F(x,(n-1)\Cdot y)
$$
is the only binary reduction of $F$, and $G_e$ is quasitrivial.
\item[(c)] If $E_F=\{e_1,e_2\}$ (with $e_1\neq e_2$), then the operations $G_{e_1},G_{e_2}\colon X^2 \to X$ defined for every $x,y\in X$ by
$$
G_{e_1}(x,y)=F(x,(n-2)\Cdot e_1,y) \quad \mbox{and} \quad G_{e_2}(x,y)=F(x,(n-2)\Cdot e_2,y)
$$
are the only binary reductions of $F$ (and $G_{e_1}\neq G_{e_2}$). Neither of $G_{e_1}$ and $G_{e_2}$ is quasitrivial  and, in this case, the identity $F((n-1)\Cdot x,y)=F(x,(n-1)\Cdot y)$ does not hold.
\end{enumerate}
\end{proposition}
Proposition \ref{prop:imp} is of particular interest since the class of associative and quasitrivial binary operations was characterized by L\"anger in \cite[Theorem 1]{Lan80}.

In this paper, we provide a characterization of the class $\mathcal{F}^n_{n-1}\setminus\mathcal{F}^n_1$. We show that all of its elements are also reducible to binary associative operations. We give a full description of the possible reductions of the operations in this class.

Let us begin with the particular case when all the elements in $X$ are neutral.
Recall that a group $(X,G)$ with neutral element $e$ has \emph{bounded exponent} if there exists an integer $m\geq 1$ such that $G^{m-1}(m\Cdot x) = e$ for any $x\in X$ (with the usual convention that $G^0(x)=x$ for every $x\in X$). In that case, the \emph{exponent} of the group is the smallest integer having this property. The following result provides a description of the class of $n$-ary semigroups containing only neutral elements. It was stated without proof in \cite[p.~2]{DudGla08} in the framework of $n$-ary groups, but it can be easily extended to $n$-ary semigroups by using \cite[Corollary~4]{DudGla08}. For the sake of completeness, we provide a direct proof that basically uses \cite[Lemma~1]{DudMuk06}.

\begin{theorem}\label{thm:main1}
Let $F\colon X^n \to X$ ($n\geq 3$) be an associative operation. Then $E_F=X$ if and only if $(X,F)$ is the $n$-ary extension of an Abelian group whose exponent divides $n-1$.
\end{theorem}

Abelian groups having bounded exponent play a central role in this first result, but also in the next theorems. We recall that Pr\"{u}fer and Baer (see, e.g., \cite[Corollary 10.37]{Rot95}) showed that if an Abelian group has bounded exponent, then it is isomorphic to a direct sum of cyclic groups. Hence, the exponent of an Abelian group divides $n-1$ if and only if the Abelian group is isomorphic to a direct sum of cyclic groups whose orders divide $n-1$.

Theorem \ref{thm:main1} also highlights the fact that an $n$-ary associative operation may have several reductions, associated with distinct neutral elements. For instance, the ternary sum on $\Z_2$ has two neutral elements, namely $0$ and $1$. It is reducible to the operations $G_0,G_1\colon \Z_2^2 \to \Z_2$ defined by $G_0(x,y) = x+y ~(\text{mod}~ 2)$ and $G_1(x,y) = x+y+1 ~(\text{mod}~ 2)$, respectively. We can also easily see that the semigroups $(\Z_2,G_0)$ and $(\Z_2,G_1)$ are isomorphic. In fact, this result can be generalized to other underlying sets: all reductions obtained in this way are isomorphic, as stated in the following result.

\begin{proposition}\label{prop:conjRed}
Let $F\colon X^n\to X$ ($n\geq 3$) be an associative operation such that $E_F\neq\varnothing$. Then every reduction of $F$ is of the form  $G_e$ for some $e\in E_F$. Moreover, if $e_1,e_2\in E_F$, then $(X,G_{e_1})$ and $(X,G_{e_2})$ are isomorphic.
\end{proposition}

In order to state one of the main results of this paper, we shall make use of the following classes of operations. Recall that an element $a\in X$ is said to be an \emph{annihilator} for $F\colon X^n\to X$ if $F(x_1,\ldots,x_n)=a$ whenever $a\in\{x_1,\ldots,x_n\}$.

\begin{definition}\label{def:H}
For every integer $m\geq 1$, let $\mathcal{H}_m$ be the class of binary operations $G\colon X^2 \to X$ such that there exists a subset $Y\subseteq X$ with $|Y|\geq 3$ for which the following assertions hold.
\begin{enumerate}
\item[(a)] $(Y,G|_{Y^2})$ is an Abelian group whose exponent divides $m$.
\item[(b)] $G|_{(X\setminus Y)^2}$ is associative and quasitrivial.
\item[(c)] Any $x\in X\setminus Y$ is an annihilator for $G|_{(\{x\}\bigcup Y)^2}$.
\end{enumerate}
\end{definition}

Note that $\mathcal{H}_1=\varnothing$. As we will see, all operations in $\mathcal{H}_m$ are associative, and the set $Y$ is unique.
In fact, the family of classes $\mathcal{H}_m$ is the key for the characterization of the classes $\mathcal{F}^n_{n-1}\setminus\mathcal{F}^n_1$.

\begin{theorem}\label{cor:cons}
Every $G\in \mathcal{H}_m$ is associative ($m\geq 1$). If $G\in \mathcal{H}_{n-1}$, then its $n$-ary extension $F=G^{n-1}$ is in $\mathcal{F}^n_{n-1}\setminus\mathcal{F}^n_1$. Conversely, for every $F\in\mathcal{F}^n_{n-1}\setminus\mathcal{F}^n_1$ we have that  $|E_F|\geq 3$, and the reductions of $F$ are exactly the operations $G_e$ for $e\in E_F$ and they lie in $\mathcal{H}_{n-1}$.
\end{theorem}

As an immediate corollary we solve the reducibility problem for operations in $\mathcal{F}^n_{n-1}$.
\begin{corollary}\label{cor:dud}
Every operation in $\mathcal{F}^n_{n-1}$ is reducible to a binary associative operation.
\end{corollary}

Theorem \ref{cor:cons} is of particular interest as it enables us to easily construct $n$-ary operations in $\mathcal{F}^n_{n-1}\setminus\mathcal{F}^n_1$. For instance, for any integers $n\geq 3$ and $p\geq 1$, the operation of the cyclic group $(\Z_n,+)$ is in $\mathcal{H}_{np}$, and thus the operation associated with its $(np+1)$-ary extension is in $\mathcal{F}^{np+1}_{np}\setminus\mathcal{F}^{np+1}_1$.

To give another example, consider the chain $(X,\leq) = (\{1,2,3,4,5\},\leq)$ together with the operation $G\colon X^2 \to X$ defined by the following conditions:
\begin{itemize}
\item $(\{1,2,3\},G|_{\{1,2,3\}^2})$ is isomorphic to $(\Z_3,+)$,
\item $G|_{\{4,5\}^2} = \vee|_{\{4,5\}^2}$, where $\vee\colon X^2 \to X$ is the maximum operation for $\leq$,
\item for any $x\in \{1,2,3\}$, $G(x,4) = G(4,x) = 4$ and $G(x,5) = G(5,x) = 5$.
\end{itemize}
Then we have $G\in \mathcal{H}_{3p}$ for any integer $p\geq 1$ and so $G^{3p}$ is in $\mathcal{F}^{3p+1}_{3p}\setminus\mathcal{F}^{3p+1}_1$.

Now we give a reformulation of Theorem \ref{cor:cons} that is not based on binary reductions.

\begin{theorem}\label{thm:main2}
If $F\in\mathcal{F}^n_{n-1}\setminus\mathcal{F}^n_1$, then, setting $Y=E_F$, we have that $|Y|\geq 3$ and the following assertions hold.
\begin{enumerate}
\item[(a)] $(Y,F|_{Y^n})$ is the $n$-ary extension of an Abelian group whose exponent divides $n-1$.
\item[(b)] $F|_{(X\setminus Y)^n}$ is associative, quasitrivial, and has at most one neutral element.
\item[(c)] For all $x_1,\ldots,x_n\in X$ and $i\in\{1,\ldots,n-1\}$ such that $\{x_i,x_{i+1}\}\cap (X\setminus Y)=\{x\}$ we have
\[F(x_1,\ldots,x_n)=F(x_1,\ldots,x_{i-1},x,x,x_{i+2},\ldots,x_n).\]
\end{enumerate}
Conversely, if an operation $F$ satisfies these conditions for some $Y\subseteq X$ with $|Y|\geq 3$, then $F\in\mathcal{F}^n_{n-1}\setminus\mathcal{F}^n_1$ and $E_F=Y$.
\end{theorem}

Proposition \ref{prop:QTDn2} shows that all operations in $\mathcal{F}^n_{n-2}$ are quasitrivial. The examples we just presented show that there are operations in $\mathcal{F}^n_{n-1}$ that are not quasitrivial, for some $n\geq 3$ and some sets $X$. Theorem \ref{cor:cons} enables us to provide necessary and sufficient conditions on the set $X$ for such operations to exist. 
\begin{definition}\label{def:cm}
 For any integer $m\geq 2$, let $c_m$ denote the cardinality of the smallest Abelian group with at least three elements whose exponent divides $m$.
\end{definition}

\begin{proposition}\label{cor:triv}
For every $n\geq 3$, we have $\mathcal{F}^n_{n-1}\setminus\mathcal{F}^n_1\neq \varnothing$ if and only if $|X|\geq c_{n-1}$.
\end{proposition}

\begin{corollary}\label{cor:triv2}
For any integer $n\geq 3$, let $p$ be the least odd prime divisor of $n-1$ if $n-1$ is not a power of $2$; otherwise, set $p=4$. The following assertions hold.
\begin{enumerate}
\item[(a)] If $n$ is even, then $\mathcal{F}^n_{n-1}\setminus\mathcal{F}^n_1\neq \varnothing$ if and only if $|X|\geq p$.
\item[(b)] If $n$ is odd, then $\mathcal{F}^n_{n-1}\setminus\mathcal{F}^n_1\neq \varnothing$ if and only if $|X|\geq \min(4,p)$.
\end{enumerate}
\end{corollary}
Finally, we observe that if $(X,\leq)$ is a semilattice that is not a chain, then the $n$-ary operation $F\colon X^n\to X$ defined by $F(x_1,\ldots,x_n)=x_1\vee \ldots\vee x_n$ is in $\mathcal{F}^n_{n}$. However, it is not in $\mathcal{F}^n_{n-1}$ since $F((n-1)\Cdot x,y)\notin\{x,y\}$ whenever $x$ and $y$ are not comparable, i.e, $x\vee y\notin\{x,y\}$. Since such a semilattice structure exists on every set $X$ such that $|X|\geq 3$, we obtain the following result.

\begin{proposition}\label{prop:easy}
For every $n\geq 2$, we have $\mathcal{F}^n_{n}\setminus\mathcal{F}^n_{n-1}\neq \varnothing$ if and only if $|X|\geq 3$.
\end{proposition}

In Section \ref{Sect:2} we give the proofs of the results above, using some more technical statements that may have interest on their own. In Section \ref{Sec:alt} we introduce and investigate an alternative hierarchy of subclasses of idempotent operations. We end the paper by some concluding remarks in Section \ref{Sec:conclusion}.

\section{Technicalities and proofs of the main results}\label{Sect:2}
Let us begin with Proposition ~\ref{prop:QTDn2}, which essentially follows from the very definition of the classes $\mathcal{F}^n_k$.
\begin{proof}[Proof of Proposition~\ref{prop:QTDn2}]
We only need to prove that $\mathcal{F}^n_{n-2}\subseteq\mathcal{F}^n_1$, and so we can assume that $n\geq 4$. Let $F\in\mathcal{F}^n_{n-2}$ and let us show by induction that for every $k\in\{1,\ldots,n\}$ we have
\begin{equation}\label{eq:propQTDn2}
F(k\Cdot x_1,x_{k+1},\ldots,x_n)\in\{x_1,x_{k+1},\ldots,x_n\},\qquad x_1,x_{k+1},\ldots,x_n\in X.
\end{equation}
By the definition of $\mathcal{F}^n_{n-2}$, condition \eqref{eq:propQTDn2} holds for any $k\in\{n-2,n-1,n\}$. Let us now assume that it holds for some $k\in\{2,\ldots,n\}$ and let us show that it still holds for $k-1$. Using associativity and idempotency, we have
\begin{eqnarray*}
F((k-1)\Cdot x_1,x_k,\ldots,x_n) &=& F(F(n\Cdot x_1),(k-2)\Cdot x_1,x_k,\ldots,x_n)\\
&=& F(k\Cdot x_1,F((n-2)\Cdot x_1,x_k,x_{k+1}),\ldots,x_n).
\end{eqnarray*}
By the induction hypothesis, the latter expression lies in $\{x_1,x_{k},\ldots,x_n\}$.

Thus, Equation \eqref{eq:propQTDn2} holds for every $k\in\{1,\ldots,n\}$. Using it for $k=1$, we obtain that $F$ is quasitrivial.
\end{proof}

In Theorem~\ref{thm:main1} and Proposition~\ref{prop:conjRed} we deal with neutral elements. We first state and prove some intermediate results concerning such elements. The following two lemmas were stated and proved in \cite[Theorem 3]{Dud01} for $n$-ary groups. We provide a proof of the first one that does not use the $n$-ary group structure but basically uses \cite[Lemma 1]{DudMuk06}, and we give a slightly different proof for the second one in the framework of $n$-ary semigroups.

\begin{lemma}\label{lem:symgeneral}
Let $F\colon X^n\to X$ be an associative operation and let $e\in E_F$. Then, for any $x_1,\ldots,x_{n-1}\in X$ we have
$$
F(x_1,\ldots,x_{n-1},e) ~=~ F(x_1,\ldots,e,x_{n-1}) ~=~ \cdots ~=~ F(e,x_1,\ldots,x_{n-1}).
$$
Moreover, for any $x\in X$ the restriction $F|_{(\{x\}\bigcup E_F)^n}$ is symmetric.
\end{lemma}
\begin{proof}
Let $x_1,\ldots,x_{n-1}\in X$ and let $G_e$ be the reduction of $F$ defined by \eqref{eq:dud}. For $i\in\{1,\ldots,n-1\}$ we have $G_e(x_i,e)=x_i=G_e(e,x_i)$, which proves the first part of the statement for $n=2$. For $n\geq 3$ we have
\[F(x_1,\ldots,x_{i},e,x_{i+1},\ldots,x_{n-1})=G_e^{n-2}(x_1,\ldots,x_{i-1},G_e(x_i,e),x_{i+1},\ldots,x_{n-1}),\]
and the first part of the statement follows from  the fact that  each $x_i$ commutes with $e$ in $G_e$. The second part is a direct consequence of the first part.
\end{proof}

\begin{lemma}\label{lem:opgeneral}
Let $F\colon X^n \to X$ be an associative operation such that $E_F\neq \varnothing$. Then $F$ preserves $E_F$, i.e., $F(E_F^n)\subseteq E_F$.
\end{lemma}

\begin{proof}
Let $e_1,\ldots,e_n\in E_F$ and let us show that $F(e_1,\ldots,e_n)\in E_F$. By Lemma \ref{lem:symgeneral} and associativity of $F$, for any $x\in X$ we have
\begin{multline*}
F((n-1)\Cdot F(e_1,\ldots,e_n),x) \\
=~ F(F(e_1,(n-1)\Cdot e_2),F(e_1,(n-1)\Cdot e_3),\ldots,F(e_1,(n-1)\Cdot e_n),x) \\
=~ F((n-1)\Cdot e_1,x) ~=~ x.
\end{multline*}
Similarly, for any $x\in X$ we can show that
$$
F(i\Cdot F(e_1,\ldots,e_n),x,(n-i-1)\Cdot F(e_1,\ldots,e_n)) ~=~ x, \qquad i\in \{0,\ldots,n-2\}.
$$
Thus $F(e_1,\ldots,e_n)$ satisfies \eqref{neutralelem}, i.e.,  $F(e_1,\ldots,e_n)\in E_F$.
\end{proof}

Combining Lemmas \ref{lem:symgeneral} and \ref{lem:opgeneral}, we immediately derive the following result.

\begin{corollary}\label{cor:mon}
If $(X,F)$ is an $n$-ary monoid, then $(E_F,F|_{E_F^n})$ is a symmetric $n$-ary monoid.
\end{corollary}

\begin{proof}[Proof of Theorem \ref{thm:main1}]
(Sufficiency) Obvious.

(Necessity) Suppose that $X=E_F$. Let $e\in E_F$ and let $G_e\colon X^2 \to X$ be the corresponding reduction of $F$ defined by \eqref{eq:dud}. Recall that $e$ is the (unique) neutral element of $G_e$ by \eqref{eq:dud}. By Corollary \ref{cor:mon}, we have that $F$ is symmetric. Thus, we have that $G_e$ also is symmetric. Moreover, since $G_e$ is a binary reduction of $F$ and $E_F=X$, it follows that
$$
G_e(G_e^{n-2}((n-1)\Cdot x),y) ~=~ y ~=~ G_e(y,G_e^{n-2}((n-1)\Cdot x)), \qquad x,y\in X,
$$
which shows that $G_e^{n-2}((n-1)\Cdot x)\in E_{G_e}$ for any $x\in X$. However, since $E_{G_e}=\{e\}$, we have that $G_e^{n-2}((n-1)\Cdot x)=e$ for any $x\in X$. Thus, $(X,G_e)$ is an Abelian group whose exponent divides $n-1$.
\end{proof}

The following result follows immediately from Theorem \ref{thm:main1}.

\begin{corollary}\label{cor:main2}
If $(X,F)$ is an $n$-ary monoid, then $(E_F,F|_{E_F^n})$ is the $n$-ary extension of an Abelian group whose exponent divides $n-1$.
\end{corollary}

\begin{proof}[Proof of Proposition~\ref{prop:conjRed}]
The first part of the statement follows from \cite[Proposition 3.3]{CouDev}.
Moreover, the associativity of $F$ and the definition of neutral elements ensure that the map $\psi\colon X\to X$ defined by
$$
\psi(x) ~=~ F(e_2,x,(n-2)\Cdot e_1)
$$
is a bijection and that $\psi^{-1}(x)=F((n-2)\Cdot e_2,x,e_1)$. We then have
\begin{eqnarray*}
\lefteqn{G_{e_2}(\psi(x),\psi(y))}\\
&=& F(F(e_2,x,(n-2)\Cdot e_1),(n-2)\Cdot e_2,F(e_2,y,(n-2)\Cdot e_1))\\
&=& F(F(e_2,x,(n-2)\Cdot e_1),F((n-1)\Cdot e_2,y),(n-2)\Cdot e_1)\\
&=& F(F(e_2,x,(n-2)\Cdot e_1),y,(n-2)\Cdot e_1)\\
&=& F(e_2,F(x,(n-2)\Cdot e_1,y),(n-2)\Cdot e_1)\\
&=& \psi(G_{e_1}(x,y)),
\end{eqnarray*}
which completes the proof.
\end{proof}

Let us now prove Theorem~\ref{cor:cons}. To this extent, we first state and prove some intermediate results. We have the following remarkable lemma, which characterizes the existence of a pair of neutral elements for $F\in\mathcal{F}_{n-1}^n$ by means of two identities.

\begin{lemma}\label{lem:charne}
Let $F\in\mathcal{F}_{n-1}^n$ and let $a,b\in X$ such that $a\neq b$. Then $a,b\in E_F$ if and only if $F((n-1)\Cdot a,b) = b$ and $F(a,(n-1)\Cdot b) = a$.
\end{lemma}

\begin{proof}
(Necessity) Obvious.

(Sufficiency) For any $x\in X$, we have
\begin{eqnarray*}
F((n-1)\Cdot a,x) &=& F((n-2)\Cdot a, F(a,(n-1)\Cdot b),x) \\
                  &=& F(F((n-1)\Cdot a,b),(n-2)\Cdot b,x) ~=~ F((n-1)\Cdot b,x),
\end{eqnarray*}
which implies that $F((n-1)\Cdot a,x)= F((n-1)\Cdot b,x)=x$ for any $x\in X$. Indeed, for $x\in\{a,b\}$ this relation follows from idempotency, and for $x\not\in\{a,b\}$ we have
$$
F((n-1)\Cdot a,x)= F((n-1)\Cdot b,x)\in\{a,x\}\cap\{b,x\}=\{x\},
$$
due to the definition of $\mathcal{F}_{n-1}^n$. Similarly, we get $F(x,(n-1)\Cdot a) = x = F(x,(n-1)\Cdot b)$ for any $x\in X$. It follows from these relations, together with associativity of $F$, that for any $k\in \{1,\ldots,n-2\}$, the maps $\psi_k,\xi_k\colon X \to X$ defined by
\begin{eqnarray*}
\psi_k(x) &=& F(k\Cdot a,x,(n-k-1)\Cdot a)\\
\xi_k(x) &=& F(k\Cdot b,x,(n-k-1)\Cdot b)
\end{eqnarray*}
are bijections with inverse maps $\psi_{n-k-1}$ and $\xi_{n-k-1}$, respectively. It then follows that, for any $k\in \{1,\ldots,n-2\}$, we have $F(k\Cdot a,x,(n-k-1)\Cdot a)=\psi_k(x)=x$ for every $x\in X$. Indeed, for $x=a$, this relation follows from idempotency, and for $x\neq a$, we have $\psi_k(x)\in\{a,x\}$ and $\psi_k(x)\neq a$. Similarly, we can show that $F(k\Cdot b,x,(n-k-1)\Cdot b)=\xi_k(x)=x$ for every $x\in X$, which shows that $a,b\in E_F$.
\end{proof}

Given an associative operation $F\colon X^n \to X$, we can define the sequence $(F^q)_{q\geq 1}$ of $(qn-q+1)$-ary associative operations inductively by the rules $F^1 = F$ and
$$
F^{q}(x_1,\ldots,x_{qn-q+1}) ~=~ F^{q-1}(x_1,\ldots,x_{(q-1)n-q+1},F(x_{(q-1)n-q+2},\ldots,x_{qn-q+1})),
$$
for any integer $q\geq 2$ and any $x_1,\ldots,x_{qn-q+1} \in X$.

The following proposition shows that every $n$-tuple that violates the quasitriviality condition for $F\in\mathcal{F}^n_{n-1}$ belongs to $E_F^n$.

\begin{proposition}\label{prop:neutregeneral}
Let $F\in\mathcal{F}^n_{n-1}$. For any $a_1,\ldots,a_n\in X$ such that
$F(a_1,\ldots,a_n)\notin\{a_1,\ldots,a_n\}$, we have that $a_1,\ldots,a_n,F(a_1,\ldots,a_n)\in E_F$. Moreover, $F|_{(X\setminus E_F)^n}$ is quasitrivial.
\end{proposition}

\begin{proof}
The case $n=2$ is trivial. So assume that $n\geq 3$.
Let us prove by induction on $k\in\{1,\ldots,n-1\}$ that for every $a_1,a_2,\ldots,a_{k+1}\in X$ the condition
$$
F((n-k)\Cdot a_1,a_2,\ldots,a_{k+1})\notin\{a_1,\ldots,a_{k+1}\}
$$
implies $a_1,\ldots,a_{k+1}\in E_F$. For $k=1$, there is nothing to prove. We thus assume that the result holds true for a given $k\in\{1,\ldots,n-2\}$ and we show that it still holds for $k+1$. Now, consider elements $a_1,\ldots, a_{k+2}$ such that
\begin{equation}\label{eq:imp}
F((n-k-1)\Cdot a_1,a_2,\ldots,a_{k+2})\notin\{a_1,\ldots,a_{k+2}\}.
\end{equation}
We first prove that $a_1,a_2\in E_F$.

If $a_1=a_2$, then $a_1,\ldots,a_{k+2}\in E_F$ by the induction hypothesis.

If $a_1\neq a_2$, then we prove that $F((n-1)\Cdot a_1,a_2)=a_2$ and $F(a_1,(n-1)\Cdot a_2)=a_1$, which show that $a_1,a_2\in E_F$ by Lemma \ref{lem:charne}.

\begin{itemize}
\item For the sake of a contradiction, assume first that $F((n-1)\Cdot a_1,a_2)=a_1$. Then, for $\ell\geq 1$ we have
\begin{eqnarray}
\lefteqn{F((n-k-1)\Cdot a_1,a_2,\ldots,a_{k+2})}\nonumber\\
&=& F^{\ell +1}(((n-k-1)+\ell (n-2))\Cdot a_1,(\ell +1)\Cdot a_2,\ldots,a_{k+2}).\label{fl1}
\end{eqnarray}
Choosing $\ell =n-k-1$ and using idempotency of $F$, we obtain
\[
F((n-k-1)\Cdot a_1,a_2,\ldots,a_{k+2}) ~=~ F^2((n-1)\Cdot a_1,(n-k)\Cdot a_2,a_3,\ldots,a_{k+2}).
\]
Since the left-hand side of this equation does not lie in $\{a_1,\ldots,a_{k+2}\}$ by \eqref{eq:imp}, we obtain
\[
F((n-k)\Cdot a_2,a_3,\ldots,a_{k+2})\notin\{a_1,\ldots,a_{k+2}\}.
\]
By the induction hypothesis, we have $a_2,\ldots,a_{k+2}\in E_F$. Then choosing $\ell = n-2$ in \eqref{fl1} and using idempotency and the fact that $a_2\in E_F$, we obtain
\begin{eqnarray*}
\lefteqn{F((n-k-1)\Cdot a_1,a_2,\ldots,a_{k+2})}\\
&=& F^{n-1}(((n-k-1)+(n-2)^2)\Cdot a_1,(n-1)\Cdot a_2,\ldots,a_{k+2})\\
&=& F^2((n-k)\Cdot a_1,(n-1)\Cdot a_2,a_3,\ldots,a_{k+2})\\
&=& F((n-k)\Cdot a_1,a_3,\ldots,a_{k+2}).
\end{eqnarray*}
By the induction hypothesis, we have $a_1\in E_F$. We then have $F((n-1)\Cdot a_1,a_2)=a_2\neq a_1$, a contradiction.
\item Assume now that $F(a_1,(n-1)\Cdot a_2)=a_2$. Then, for $\ell\geq 1$ we have
\begin{eqnarray*}
\lefteqn{F((n-k-1)\Cdot a_1,a_2,\ldots,a_{k+2})}\nonumber\\
&=& F^{\ell +1}((n-k-1+\ell )\Cdot a_1,(\ell (n-2)+1)\Cdot a_2,\ldots,a_{k+2}).\label{fl2}
\end{eqnarray*}
For $\ell =k$, using idempotency and the fact that $k(n-2)+1=n-k+(k-1)(n-1)$, we obtain
\begin{eqnarray*}
\lefteqn{F((n-k-1)\Cdot a_1,a_2,\ldots,a_{k+2})}\\
&=& F^2((n-1)\Cdot a_1,(n-k)\Cdot a_2,a_3,\ldots,a_{k+2}).
\end{eqnarray*}
Thus, $F((n-k)\Cdot a_2,a_3,\ldots,a_{k+2})\notin\{a_1,\ldots,a_{k+2}\}$. By the induction hypothesis, we have $a_2,\ldots,a_{k+2}\in E_F$.
It follows that $F(a_1,(n-1)\Cdot a_2)=a_1\neq a_2$, a contradiction.
\end{itemize}
Now, since $a_2\in E_F$, it commutes with all other arguments of $F$ by Lemma \ref{lem:symgeneral}. Also, by \eqref{eq:imp} we have
$$
F((n-k-1)\Cdot a_1,a_3,\ldots,a_{k+2},a_2) \notin \{a_1,\ldots,a_{k+2}\},
$$
and thus $a_3\in E_F$. Repeating this argument, we have that $a_1,\ldots,a_{k+2}\in E_F$.

It follows from the induction that if $F(a_1,\ldots,a_n)\notin \{a_1,\ldots,a_n\}$, then $a_1,\ldots,a_n\in E_F$. Finally we have $F(a_1,\ldots,a_n)\in E_F$ by Lemma \ref{lem:opgeneral}. The second part is straightforward.
\end{proof}

Proposition \ref{prop:imp} shows that a quasitrivial $n$-ary semigroup cannot have more than two neutral elements. The next result shows that an operation in $\mathcal{F}^n_{n-1}$ is quasitrivial whenever it has at most two neutral elements.

\begin{corollary}\label{cor:quasi}
An operation $F\in\mathcal{F}^n_{n-1}$ is quasitrivial if and only if $|E_F|\leq 2$.
\end{corollary}

\begin{proof}
(Necessity) This follows from Proposition \ref{prop:imp}.

(Sufficiency) Suppose that $F$ is not quasitrivial, i.e., there exist $a_1,\ldots,a_n\in X$ such that $F(a_1,\ldots,a_n)\notin\{a_1,\ldots,a_n\}$. Since $F$ is idempotent, we must have $|\{a_1,\ldots,a_n\}|\geq 2$ and so $|\{a_1,\ldots,a_n,F(a_1,\ldots,a_n)\}|\geq 3$. We also have $\{a_1,\ldots,a_n,F(a_1,\ldots,a_n)\}\subseteq E_F$ by Proposition \ref{prop:neutregeneral}. Therefore we have $|E_F|\geq 3$.
\end{proof}

\begin{proposition}\label{prop:anngeneral}
Let $F\in\mathcal{F}^n_{n-1}$ and suppose that $|E_F|\geq 3$. Then, any element $x\in X\setminus E_F$ is an annihilator of $F|_{(\{x\}\bigcup E_F)^n}$. Moreover, $F|_{(X\setminus E_F)^n}$ is quasitrivial and has at most one neutral element.
\end{proposition}

\begin{proof}
Let $x\in X\setminus E_F$ and $e\in E_F$ and let us show that $F(k\Cdot x, (n-k)\Cdot e) = x$ for any $k\in \{1,\ldots,n-1\}$. If $k=1$, then this equality follows from the definition of a neutral element. Now, suppose that there exists $k\in \{2,\ldots,n-1\}$ such that $F(k\Cdot x,(n-k)\Cdot e) \neq x$. Since $x\in X\setminus E_F$, by Proposition~\ref{prop:neutregeneral} we must have $F(k\Cdot x, (n-k)\Cdot e) = e$. But then, using the associativity of $F$, we get
\begin{eqnarray*}
F((n-1)\Cdot x,e) &=& F((n-1)\Cdot x,F(k\Cdot x,(n-k)\Cdot e))\\
&=& F(k\Cdot x,(n-k)\Cdot e) ~=~ e,
\end{eqnarray*}
and we conclude by Lemma~\ref{lem:charne} that $x\in E_F$, which contradicts our assumption. Thus, we have
\begin{equation}\label{eq:need1}
F(k\Cdot x, (n-k)\Cdot e) ~=~ x, \qquad k\in \{1,\ldots,n-1\}.
\end{equation}
Now, let us show that $F(k\Cdot x,e_{k+1},\ldots,e_{n}) = x$ for any $k\in \{1,\ldots,n-1\}$ and any $e_{k+1},\ldots,e_{n}\in E_F$. To this extent, we only need to show that
$$
F(k\Cdot x,e_{k+1},\ldots,e_{n}) ~=~ F((k+1)\Cdot x,e_{k+2},\ldots,e_{n}),
$$
for any $k\in \{1,\ldots,n-1\}$ and any $e_{k+1},\ldots,e_{n}\in E_F$.
So, let $k\in \{1,\ldots,n-1\}$ and $e_{k+1},\ldots,e_{n}\in E_F$. Using \eqref{eq:need1} and the associativity of $F$ we get
\begin{eqnarray*}
F(k\Cdot x,e_{k+1},\ldots,e_{n})
&=& F((k-1)\Cdot x,F(2\Cdot x,(n-2)\Cdot e_{k+1}),e_{k+1},\ldots,e_{n}) \\
&=& F(k\Cdot x,F(x,(n-1)\Cdot e_{k+1}),e_{k+2},\ldots,e_{n}) \\
&=& F((k+1)\Cdot x,e_{k+2},\ldots,e_{n}),
\end{eqnarray*}
which completes the proof by idempotency of $F$ and Lemma~\ref{lem:symgeneral}. For the second part of the proposition, we observe that $F|_{(X\setminus E_F)^n}$ is quasitrivial by Proposition~\ref{prop:neutregeneral}. Also, using \eqref{eq:need1} and the associativity of $F$, for any $x,y\in X\setminus E_F$ and any $e\in E_F$ we obtain
\begin{eqnarray*}
F((n-1)\Cdot x,y) &=& F((n-1)\Cdot x,F(e,(n-1)\Cdot y)) \\
                  &=& F(F((n-1)\Cdot x,e),(n-1)\Cdot y) ~=~ F(x,(n-1)\Cdot y),
\end{eqnarray*}
which shows that $F|_{(X\setminus E_F)^n}$ cannot have more than one neutral element.
\end{proof}

\begin{proof}[Proof of Theorem \ref{cor:cons}]
In order to show that every $G\in\mathcal{H}_m$ is associative, we have to compare the expressions $G(G(x_1,x_2),x_3)$ and $G(x_1,G(x_2,x_3))$ for all $x_1,x_2,x_3$ in $X$. Clearly these expressions are equal if all their arguments are either in $Y$ or in $X\setminus Y$ since the restriction of $G$ to these subsets is associative. If one argument, say $x_i$, is in $X\setminus Y$ and the others are in $Y$, then both expressions are equal to $x_i$ by Property (c) of Definition \ref{def:H}. For the same reason, if the arguments $x_i,x_j$ are in $X\setminus Y$ and the the third one in $Y$, then both expressions are equal to $G(x_i,x_j)$.

Now, we consider $G\in\mathcal{H}_{n-1}$ and define $F=G^{n-1}$. Then we have $E_F=Y$. Indeed, conditions (a) and (c) of Definition \ref{def:H} imply directly that $Y\subseteq E_F$. Moreover if $x\notin Y$, then still by condition (c) we have $F((n-1)\Cdot x,y)=x\neq y$ for $y\in Y$, so $x\notin E_F$.

Next, we show that $F(k\Cdot x,y,(n-k-1)\Cdot x)\in\{x,y\}$ for every $x,y\in X$. If $x\in Y$, $x$ is a neutral element, so this expression is equal to $y$. If $x\in X\setminus Y$, then either $y\in Y$ and this expression is equal to $x$ (by condition (c)), or $y\in X\setminus Y$, and this expression is in $\{x,y\}$ (by condition (b)). Finally, $F\notin\mathcal{F}^n_1$ by Corollary \ref{cor:quasi}, since $|E_F|=|Y|\geq 3$.

Now we prove the converse statement and consider $F\in\mathcal{F}^n_{n-1}\setminus \mathcal{F}^n_{1}$. Setting $Y=E_F$ we have $|Y|\geq 3$ by Corollary \ref{cor:quasi}. By Proposition \ref{prop:conjRed}, every reduction of $F$ is of the form $G_e$ for some $e\in E_F$.

Finally, we show that $G_e\in\mathcal{H}_{n-1}$. We have that $(Y,G_e|_{Y^2})$ is an Abelian group whose exponent divides $n-1$ by Corollary \ref{cor:main2} and Proposition \ref{prop:conjRed}. Also, we have that $G_e|_{(X\setminus Y)^2}$ is quasitrivial by Propositions \ref{prop:imp} and \ref{prop:anngeneral}. Finally, we have that any $x\in X\setminus Y$ is an annihilator for $G_e|_{(\{x\}\bigcup Y)^2}$ by Proposition \ref{prop:anngeneral}.
\end{proof}

\begin{proof}[Proof of Corollary \ref{cor:dud}]
This follows from Proposition \ref{prop:imp} and Theorem \ref{cor:cons}.
\end{proof}

\begin{remark}
In the proof of Corollary \ref{cor:dud} we used \cite[Corollary 3.11]{CouDev} which is based on results obtained by Ackerman \cite{Ack}. In the appendix we provide an alternative proof of Corollary \ref{cor:dud} that does not make use of \cite[Corollary 3.11]{CouDev}.
\end{remark}

\begin{proof}[Proof of Theorem \ref{thm:main2}]
If $F\in\mathcal{F}^n_{n-1}\setminus\mathcal{F}^n_1$, then by Theorem \ref{cor:cons} we have $|E_F|\geq 3$ and for every $e\in E_F$, $G_e$ is in $\mathcal{H}_{n-1}$. Then $G_e|_{Y^2}$ is a reduction of $F|_{Y^n}$ and (a) holds true. Also $G_e|_{(X\setminus Y)^2}$ is a quasitrivial reduction of $F|_{(X\setminus Y)^n}$, so (b) holds true by Proposition \ref{prop:imp}. Finally, if $x_i,x_{i+1}$ satisfy the conditions of (c), we have $G_e(x_i,x_{i+1})=x=G_e(x,x)$, so that (c) holds true.

Let us now assume that an operation $F$ satisfies conditions (a), (b), and (c). By (a), there exists an Abelian group $(Y,G_Y)$ whose exponent divides $n-1$ such that $(Y,F|_{Y^n})$ is the $n$-ary extension of $(Y,G_Y)$. We denote by $e$ the neutral element of $G_Y$. We also define the operation $G\colon X^2\to X$ by $G(x,y)=F(x,(n-2)\Cdot e,y)$ for every $x,y\in X$. We now show that $G$ is in $\mathcal{H}_{n-1}$. It is easy to see that $G|_{Y^2}=G_Y$. Then, by condition (c), $G|_{(X\setminus Y)^2}(x,y)=F((n-1)\Cdot x,y)$, so $G|_{(X\setminus Y)^2}$ is the unique quasitrivial reduction of $F|_{(X\setminus Y)^n}$ (see Proposition \ref{prop:imp}). Finally, condition (c) also implies that any $x\in X\setminus Y$ is an annihilator for $G|_{(\{x\}\bigcup Y)^2}$. Then by Theorem \ref{cor:cons}, $G$ is associative and we have $G^{n-1}\in\mathcal{F}^n_{n-1}\setminus\mathcal{F}^n_1$. We conclude the proof by showing that $G^{n-1}=F$. To this aim we compare $G^{n-1}(x_1,\ldots,x_n)$ and $F(x_1,\ldots,x_n)$ for every $(x_1,\ldots,x_n)\in X^n$. We already showed that both expressions coincide if $(x_1,\ldots,x_n)$ belongs to $Y^n$ or $(X\setminus Y)^n$. Otherwise, let us denote by $\sigma_1,\ldots,\sigma_r$ the integers such that $\{i : x_i\in X\setminus Y\} = \{\sigma_1,\ldots,\sigma_r\}$ and $\sigma_1 < \cdots < \sigma_r$. By condition (c) there exist integers $a_1,\ldots,a_r$ such that \[F(x_1,\ldots,x_n)=F(a_1\Cdot x_{\sigma_1},\ldots,a_r\Cdot x_{\sigma_r}).\]
This expression is equal to $G^{r-1}(x_{\sigma_1},\ldots,x_{\sigma_r})$ because $G|_{(X\setminus Y)^2}$ is a quasitrivial reduction of $F|_{(X\setminus Y)^n}$. Using condition (c) in Definition \ref{def:H} for $G\in\mathcal{H}_{n-1}$ we get that this expression is equal to $G^{n-1}(x_1,\ldots,x_n)$.
\end{proof}

\begin{proof}[Proof of Proposition \ref{cor:triv}]
If $\mathcal{F}^n_{n-1}\setminus\mathcal{F}^n_1\neq \varnothing$, then Theorem \ref{cor:cons} implies that there is a subset $Y\subseteq X$ and an Abelian group $(Y,G)$ whose exponent divides $n-1$ and $|Y|\geq 3$. This shows that $|X|\geq |Y| \geq c_{n-1}$. Conversely, assume that $|X|\geq c_{n-1}$. Then we choose a subset $Y\subseteq X$ such that $|Y| = c_{n-1}\geq 3$ and we endow $Y$ with an operation $G_Y$ such that $(Y,G_Y)$ is an Abelian group whose exponent divides $n-1$.

Let us consider the operation $G\colon X^2 \to X$ defined by the conditions that any $x\in X\setminus Y$ is an annihilator for $G|_{(\{x\}\bigcup Y)^2}$, that $G|_{Y^2} = G_Y$, and that $G(x,y)=y$ for any $x,y\in X\setminus Y$.
Then we have $G\in \mathcal{H}_{n-1}$ and so $G^{n-1} \in \mathcal{F}^n_{n-1}\setminus\mathcal{F}^n_1$ by Theorem \ref{cor:cons}, which concludes the proof.
\end{proof}

\begin{proof}[Proof of Corollary \ref{cor:triv2}]
By Proposition \ref{cor:triv} it is sufficient to compute $c_{n-1}$ in the two cases.
\begin{enumerate}
\item[(a)] The cyclic group of order $p$ is an Abelian group with at least three elements whose exponent divides $n-1$, hence $c_{n-1}\leq p$. On the other hand, let $(Y,G)$ be any Abelian group with at least three elements whose exponent $m$ divides $n-1$. Let $q$ be a prime divisor of $m$; then $q$ divides $n-1$, hence $q$ is odd. From the definition of the exponent it follows that $Y$ contains an element of order $q$, thus $|Y|\geq q$. Since $q$ divides $n-1$, we have $q\geq p$ by the minimality of $p$. Therefore, $|Y|\geq q\geq p$, which shows that $c_{n-1}\geq p$.
\item[(b)] If $p=3$, then we can take the group $\mathbb{Z}_p$ as in the previous case; if $p\geq 5$, then we can take the group $\mathbb{Z}_2^2$ (with exponent $2$ dividing $n-1$) in order to see that $c_{n-1}\leq \min(4,p)$. Conversely, let $(Y,G)$ be any Abelian group with at least three elements and with exponent $m$ such that $m$ divides $n-1$. If $m$ has an odd prime divisor $q$, then we can conclude that $|Y|\geq q\geq p\geq \min(4,p)$ just as in $(a)$. If $m$ has no odd prime divisors, then $m$ is a power of $2$, and then $|Y|$ is even, which together with $|Y|\geq 3$ implies that $|Y|\geq 4\geq \min(4,p)$. Thus, we conclude that $c_{n-1}\geq \min(4,p)$.\qedhere
\end{enumerate}
\end{proof}

\section{An alternative hierarchy}\label{Sec:alt}

For any integer $k\geq 1$, let $S^n_k$ be the set of $n$-tuples $(x_1,\ldots,x_n)\in X^n$ such that $|\{x_1,\ldots,x_n\}| \leq k$. Of course, we have $D^n_k \subseteq S^n_{n-k+1}$ for $k\in \{1,\ldots,n\}$. Also, we have $S^n_k \subseteq S^n_{k+1}$ for $k\in \{1,\ldots,n-1\}$. Now, denote by $\mathcal{G}^n_k$ the class of those associative $n$-ary operations $F\colon X^n\to X$ satisfying
$$
F(x_1,\ldots,x_n)\in\{x_1,\ldots,x_n\},\qquad (x_1,\ldots,x_n)\in S^n_k.
$$
We say that these operations are \emph{quasitrivial on $S^n_k$}.

It is not difficult to see that if $F\in \mathcal{G}^n_k$, then $F\in \mathcal{F}_{n-k+1}^n$. Actually, we have $\mathcal{G}^n_1=\mathcal{F}_{n}^n$ and $\mathcal{G}^n_n=\mathcal{F}_{1}^n$. These are the only classes when $n=2$, and thus we assume throughout this section that $n\geq 3$. Due to Proposition \ref{prop:QTDn2}, we have that $\mathcal{G}^n_n = \cdots = \mathcal{G}^n_3$ is exactly the class of quasitrivial associative $n$-ary operations, and hence we only need to consider operations in $\mathcal{G}^n_2$. The counterpart of Theorem~\ref{cor:cons} can then be stated as follows.

\begin{theorem}\label{thm:main3}
If $n$ is odd and $G\in \mathcal{H}_{2}$, then its $n$-ary extension $F=G^{n-1}$ is in $\mathcal{G}^n_{2}\setminus\mathcal{G}^n_n$. Conversely, for every $F\in\mathcal{G}^n_{2}\setminus\mathcal{G}^n_n$ we have $|E_F|\geq 3$, $n$ is odd, the reductions of $F$ are exactly the operations $G_e$ for $e\in E_F$, and they lie in $\mathcal{H}_{2}$.
\end{theorem}
\begin{proof}
If $n$ is odd and $G\in \mathcal{H}_{2}$, then $n-1$ is even, and so $G\in \mathcal{H}_{n-1}$. Therefore by Theorem~\ref{cor:cons}, $F=G^{n-1}$ is in $\mathcal{F}^n_{n-1}\setminus \mathcal{F}^n_{1}=\mathcal{F}^n_{n-1}\setminus \mathcal{G}^n_{n}$. We have shown in the proof of Theorem \ref{cor:cons} that $E_F=Y$. In order to show $F\in\mathcal{G}^n_{2}$, we need to show that if $x_1,\ldots,x_n\in\{x,y\}$, then $F(x_1,\ldots,x_n)\in\{x,y\}$. If $x$ or $y$ is in $X\setminus Y$, this follows from Proposition \ref{prop:anngeneral}. If $\{x,y\}\subseteq Y$, then if $k$ arguments are equal to $x$ and $n-k$ are equal to $y$, $F(x_1,\ldots,x_n)=F(k\Cdot x,(n-k)\Cdot y)$ because $(Y,G|_{Y^2})$ is an Abelian group. Since $n$ is odd, the parity of $k$ and of $n-k$ are different. Since $(Y,G|_{Y^2})$ has exponent 2, this expression is equal to $x$ (resp. $y$) when $k$ is odd (resp. even).

Conversely, if $F\in\mathcal{G}^n_{2}\setminus\mathcal{G}^n_n\subseteq \mathcal{F}^n_{n-1}\setminus\mathcal{F}^n_1$, then by Theorem~\ref{cor:cons}, we have $|E_F|\geq 3$, all the reductions of $F$ are exactly the operations $G_e$ for $e\in E_F$ and they lie in $\mathcal{H}_{n-1}$. In particular, for any $e\in E_F$, we have that $(E_F,G_e)$ is an Abelian group whose exponent divides $n-1$. However, since the neutral element is the only idempotent element of a group and since $G_e(e',e')\in \{e,e'\}$ for any $e,e'\in E_F$, it follows that $G_e(e',e') = e$ for any $e,e'\in E_F$, i.e., for any $e\in E_F$ we have that $(E_F,G_e)$ is a group of exponent $2$. Therefore, we conclude that $(E_F,F|_{E_F^n})$ is the $n$-ary extension of an Abelian group of exponent $2$. Also, since $2$ divides $n-1$ we conclude that $n$ is odd.
\end{proof}
Theorem \ref{thm:main3} is particularly interesting as it enables us to construct easily $n$-ary operations in $\mathcal{G}^n_{2}\setminus\mathcal{G}^n_n$. For instance, consider the set $X = \{1,2,3,4,5,6\}$ together with the operation $G\colon X^2 \to X$ defined by the following conditions:
\begin{itemize}
\item $(\{1,2,3,4\},G|_{\{1,2,3,4\}^2})$ is isomorphic to $(\Z_2^2,+)$,
\item $G|_{\{5,6\}^2} = \pi_1|_{\{5,6\}^2}$, where $\pi_1\colon X^2 \to X$ is defined by $\pi_1(x,y)=x$ for any $x,y\in X$,
\item for any $x\in \{1,2,3,4\}$, $G(x,5)=G(5,x)=5$ and $G(x,6)=G(6,x)=6$.
\end{itemize}
Then for any integer $p\geq 1$, we have that the operation associated with any $(2p+1)$-ary extension of $(\{1,2,3,4,5,6\},G)$ is in $\mathcal{G}^{2p+1}_{2}\setminus\mathcal{G}^{2p+1}_{2p+1}$ by Theorem \ref{thm:main3}.

We now state a reformulation of Theorem \ref{thm:main3} that does not make use of binary reductions. We omit the proof of this result as it is a straightforward adaptation of the proof of Theorem \ref{thm:main2}.
\begin{theorem}\label{cor:cons2}
If an operation $F$ is in $\mathcal{G}^n_{2}\setminus\mathcal{G}^n_n$, then $n$ is odd and setting $Y=E_F$ we have $|Y|\geq 3$ and the following assertions hold.
\begin{enumerate}
\item[(a)] $(Y,F|_{Y^n})$ is the $n$-ary extension of an Abelian group of exponent $2$.
\item[(b)] $F|_{(X\setminus Y)^n}$ is associative, quasitrivial, and has at most one neutral element.
\item[(c)] For all $x_1,\ldots,x_n\in X$ and $i\in\{1,\ldots,n-1\}$ such that $\{x_i,x_{i+1}\}\cap (X\setminus Y)=\{x\}$ we have
\[F(x_1,\ldots,x_n)=F(x_1,\ldots,x_{i-1},x,x,x_{i+2},\ldots,x_n).\]
\end{enumerate}
Conversely, if $n$ is odd and $F$ is an operation that satisfies these conditions for some $Y\subseteq X$ such that $|Y|\geq 3$, then $F\in\mathcal{G}^n_{2}\setminus\mathcal{G}^n_n$ and $E_F=Y$.
\end{theorem}

We end this section with the counterpart of Proposition \ref{cor:triv} and Corollary \ref{cor:triv2} for operations in $\mathcal{G}^n_{2}\setminus\mathcal{G}^n_n$.

\begin{corollary}
We have $\mathcal{G}^n_{2}\setminus\mathcal{G}^n_n \neq \varnothing$ if and only if $n$ is odd and $|X|\geq 4$.
\end{corollary}

\begin{proof}
(Necessity) By Theorem \ref{thm:main3}, we have that $n$ is odd and there exists a subset $Y\subseteq X$ and an Abelian group $(Y,G)$ of exponent $2$ such that $|Y|\geq 3$. Since $(Y,G)$ is of exponent $2$ we have $|X|\geq |Y|\geq 4$.

(Sufficiency) Let $Y\subseteq X$ such that $|Y| = 4$. We can endow $Y$ with an operation $G_Y$ such that $(Y,G_Y)$ is an Abelian group of exponent 2 that is isomorphic to $(\Z_2^2,+)$. Let us consider the operation $G\colon X^2 \to X$ defined by the following conditions:
\begin{itemize}
\item $G|_{Y^2} = G_Y$.
\item $G(x,y) = y$ for any $x,y\in X\setminus Y$.
\item  Any $x\in X\setminus Y$ is an annihilator for $G|_{(\{x\}\bigcup Y)^2}$.
\end{itemize}
It is not difficult to see that $G\in \mathcal{H}_2$ (see Definition \ref{def:H}). Thus, we have $G^{n-1} \in \mathcal{G}^n_{2}\setminus\mathcal{G}^n_n$ by Theorem \ref{thm:main3}, which concludes the proof.
\end{proof}

\section{Concluding remarks}\label{Sec:conclusion}

In this paper we characterized the class $\mathcal{F}^n_{n-1}\setminus\mathcal{F}^n_1$, i.e., the class of those associative operations $F\colon X^n \to X$ that are not quasitrivial but satisfy the condition $F(x_1,\ldots,x_n)\in \{x_1,\ldots,x_n\}$ whenever at least $n-1$ of the elements $x_1,\ldots,x_n$ are equal to each other (Theorems \ref{cor:cons} and \ref{thm:main2}). These characterizations enabled us to obtain necessary and sufficient conditions on the cardinality of $X$ so that $\mathcal{F}^n_{n-1}\setminus\mathcal{F}^n_1\neq \varnothing$. Moreover, we proved that any operation in $\mathcal{F}^n_{n-1}\setminus\mathcal{F}^n_1$ is reducible to a binary associative operation (Corollary \ref{cor:dud}). Finally, we characterized the class $\mathcal{G}^n_{2}\setminus\mathcal{G}^n_n$, i.e., the class of those associative operations $F\colon X^n \to X$ that are not quasitrivial but satisfy the condition $F(x_1,\ldots,x_n)\in \{x_1,\ldots,x_n\}$ whenever $|\{x_1,\ldots,x_n\}|\leq 2$ (Theorems \ref{thm:main3} and \ref{cor:cons2}). As a byproduct of these characterizations, we obtained necessary and sufficient conditions on the cardinality of $X$ for which $\mathcal{G}^n_{2}\setminus\mathcal{G}^n_n\neq \varnothing$.

The main results of this paper thus characterize several relevant subclasses of associative and idempotent $n$-ary operations. However, the characterization of the class $\mathcal{F}^n_n$ of associative and idempotent $n$-ary operations still eludes us. This and related enumeration results \cite{CouDev,CouDevMar} constitute a topic of current research.

\section*{Acknowledgments}

The authors would like to thank the anonymous reviewers for their insightful remarks that helped improving the current paper. They are especially grateful for Proposition \ref{cor:triv} and Corollary \ref{cor:triv2}. The second author is supported by the Luxembourg National Research Fund under the project PRIDE 15/10949314/GSM.

\appendix\section{Alternative proof of Corollary \ref{cor:dud}}

We provide an alternative proof of Corollary \ref{cor:dud} that does not use \cite[Corollary 2.3]{CouDev}.

To this extent, we first prove the following general result.

\begin{proposition}\label{prop:red1}
Let $F \in \mathcal{F}^n_n$. The following assertions are equivalent.
\begin{enumerate}
\item[(i)] $F$ is reducible to an associative and idempotent operation $G\colon X^2 \to X$.
\item[(ii)] $F((n-1)\Cdot x,y) = F(x,(n-1)\Cdot y)$ for any $x,y\in X$.
\end{enumerate}
\end{proposition}

\begin{proof}
The implication (i) $\Rightarrow$ (ii) is straightforward. Now, let us show that (ii) implies (i). So, suppose that
\begin{equation}\label{eq:sym}
F((n-1)\Cdot x,y) ~=~ F(x,(n-1)\Cdot y) \qquad x,y\in X,
\end{equation}
and consider the operation $G\colon X^2 \to X$ defined by $G(x,y) = F((n-1)\Cdot x,y)$ for any $x,y\in X$. It is not difficult to see that $G$ is associative and idempotent. Now, let $x_1,\ldots,x_n\in X$ and let us show that $G^{n-1}(x_1,\ldots,x_n)=F(x_1,\ldots,x_n)$. Using repeatedly \eqref{eq:sym} and the idempotency of $F$ we obtain
\begin{eqnarray*}
G^{n-1}(x_1,\ldots,x_n) &=& F^{n-1}((n-1)\Cdot x_1,(n-1)\Cdot x_2,\ldots,(n-1)\Cdot x_{n-1},x_n) \\
&=& F^{n-1}((2n-3)\Cdot x_1,x_2,(n-1)\Cdot x_3,\ldots,(n-1)\Cdot x_{n-1},x_n) \\
&=& \cdots \\
&=& F^{n-1}(((n-2)(n-1)+1)\Cdot x_1,x_2,x_3,\ldots,x_{n-1},x_n)\\
&=& F(x_1,\ldots,x_n),
\end{eqnarray*}
which shows that $F$ is reducible to $G$.
\end{proof}

\begin{remark}
Let $\leq$ be a total ordering on $X$. An operation $F\colon X^n \to X$ is said to be \emph{$\leq$-preserving} if $F(x_1,\ldots,x_n) \leq F(x_1',\ldots,x_n')$ whenever $x_i\leq x_i'$ for any $i\in \{1,\ldots,n\}$.
One of the main results of Kiss and Somlai \cite[Theorem 4.8]{KiSom18} is that every $\leq$-preserving operation $F\in \mathcal{F}^n_n$ is reducible to an associative, idempotent, and $\leq$-preserving binary operation. To this extent, they first show \cite[Lemma 4.1]{KiSom18} that any $\leq$-preserving operation $F\in \mathcal{F}^n_n$ satisfies
\begin{equation*}
F((n-1)\Cdot x,y) ~=~ F(x,(n-1)\Cdot y) \qquad x,y\in X.
\end{equation*}
Thus, we conclude that \cite[Theorem 4.8]{KiSom18} is an immediate consequence of \cite[Lemma 4.1]{KiSom18} and Proposition \ref{prop:red1} above.
\end{remark}

The following result is the key for the alternative proof of Corollary \ref{cor:dud}.

\begin{proposition}\label{prop:red2}
Let $F \in \mathcal{F}^n_{n-1}$. The following assertions are equivalent.
\begin{enumerate}
\item[(i)] $F$ is reducible to an associative and quasitrivial operation $G\colon X^2 \to X$.
\item[(ii)] $F$ is reducible to an associative and idempotent operation $G\colon X^2 \to X$.
\item[(iii)] $F((n-1)\Cdot x,y) = F(x,(n-1)\Cdot y)$ for any $x,y\in X$.
\item[(iv)] $|E_F|\leq 1$.
\end{enumerate}
\end{proposition}

\begin{proof}
The equivalence (i) $\Leftrightarrow$ (ii) and the implication (iii) $\Rightarrow$ (iv) are straightforward. Also, the equivalence (ii) $\Leftrightarrow$ (iii) follows from Proposition \ref{prop:red1}. Now, let us show that (iv) implies (iii). So, suppose that $|E_F|\leq 1$ and suppose to the contrary that there exist $x,y\in X$ with $x\neq y$ such that $F((n-1)\Cdot x,y) \neq F(x,(n-1)\Cdot y)$. We have two cases to consider. If $F((n-1)\Cdot x,y)=y$ and $F(x,(n-1)\Cdot y)=x$, then by Lemma 2.5 we have that $x,y\in E_F$, which contradicts our assumption on $E_F$. Otherwise, if $F((n-1)\Cdot x,y)=x$ and $F(x,(n-1)\Cdot y)=y$, then we have
\begin{multline*}
x ~=~ F((n-1)\Cdot x,y) ~=~ F((n-1)\Cdot x,F(n\Cdot y))\\
  ~=~ F(F((n-1)\Cdot x,y),(n-1)\Cdot y) ~=~ F(x,(n-1)\Cdot y) ~=~y,
\end{multline*}
which contradicts the fact that $x\neq y$.
\end{proof}

\begin{proof}[Proof of Corollary \ref{cor:dud}]
This follows from Proposition \ref{prop:red2} and \cite[Lemma 1]{DudMuk06}.
\end{proof}

\end{document}